\documentclass{article}
\usepackage{amsmath}
\usepackage{amssymb}
\usepackage{amsfonts}
\usepackage{amsthm}

\def\reals{\mathbb{R}}

\def\uball{\mathbb{B}}
\def\ereals{\overline{\mathbb{R}}}

\def\comp{\raise 1pt \hbox{$\scriptstyle\circ$}}

\def\argmin{\mathop{\rm argmin}\limits}

\def\minimize{\mathop{\rm minimize}\limits}
\def\maximize{\mathop{\rm maximize}\limits}

\def\st{\mathop{\rm subject\ to}}

\def\dom{\mathop{\rm dom}\nolimits}

\def\ovr{\mathop{\rm over}}
\def\upto{{\raise 1pt \hbox{$\scriptstyle \,\nearrow\,$}}}
\def\downto{{\raise 1pt \hbox{$\scriptstyle \,\searrow\,$}}}

\def\tos{\rightrightarrows}
\def\FF{(\F_t)_{t=0}^T}

\def\B{{\cal B}}

\def\F{{\cal F}}
\def\G{{\cal G}}

\def\N{{\cal N}}

\newtheorem{theorem}{Theorem}
\newtheorem{lemma}[theorem]{Lemma}

\newtheorem{example}{Example}
\newtheorem{remark}{Remark}
\theoremstyle{definition}

\newtheorem{assumption}{Assumption}

\title{Non-convex dynamic programming and optimal investment}

\author{Teemu Pennanen\thanks{Department of Mathematics, King's College London,
Strand, London, WC2R 2LS, United Kingdom} \and 
Ari-Pekka Perkki\"o\thanks{Department of Mathematics, Technische Universit\"at Berlin, Building MA,
Str. des 17. Juni 136, 10623 Berlin, Germany} \and 
Mikl\'os R\'asonyi\thanks{MTA Alfr\'ed
R\'enyi Institute of Mathematics, Re\'altanoda utca 13--15, 1053 Budapest, Hungary and 
P\'azm\'any P\'eter Catholic University, Budapest}}

\begin{document}

\maketitle

\begin{abstract}
We establish the existence of minimizers in a rather general setting of dynamic stochastic optimization without assuming either convexity or coercivity of the objective function. We apply this to prove the existence of optimal portfolios for non-concave utility maximization problems in financial market models with frictions (such as illiquidity), a first result of its kind. The proofs are based on the dynamic programming principle whose validity is established under quite general assumptions.
\end{abstract}


\section{Introduction}

We study stochastic optimization problems in finite discrete time. The novelty is that we prove the validity
of the dynamic programming principle and the existence of optimal strategies in
cases where the objective function fails to be convex or coercive (Theorem \ref{thm:dp}). Our main result
extends the existence result of \cite{evs76} by relaxing the compactness assumption and that of 
\cite{pp12} by relaxing the assumption of convexity.

Our motivation comes mostly from mathematical finance. In standard optimal investment problems an agent tries to maximise her expected utility over available investment strategies. Utility functions are assumed concave in the overwhelming majority of the economics literature, starting already with \cite{bernoulli}. This feature is usually justified by the \emph{risk aversion} of the agents; see e.g.\ \cite{kre88} or \cite[Chapter~2]{fs}. However, the alternative theory of \cite{tk,kt} considered so-called ``$S$-shaped'' utilities (which are convex up to a certain point and concave beyond it). They also argued that investors distort objective probabilities in their decision-making procedures. 

There has been growing interest in non-concave utilities recently. Due to the mathematical difficulties, 
however, continuous-time studies focussed on the (rather unrealistic) case of complete markets 
where every contingent claim can be replicated; see \cite{cp,bkp,jz,cd11,cd,reichlin,rr}. In discrete time frictionless models also incomplete markets have been treated: one-step models were investigated in \cite{hz,bg} and multistep ones in \cite{cr,nc,miki}. All these papers assumed frictionless financial markets. 

According to our knowledge, all existing results on optimal investment under frictions (transaction costs, illiquidity effects, etc.) assume a concave utility function; see e.g.\ \cite{gura,pen14,cs14} and the references therein. 
In \cite{gura} a general, continuous-time existence result was obtained 
under the assumption that trading 
costs are superlinear functions of the trading speed. In the analogous discrete-time multiperiod setting, 
Theorem \ref{marquis} below provides an existence result for optimal investment in illiquid markets and with not necessarily concave utilities. To the best of our knowledge, this is the first result involving non-concave utilities in markets with frictions.

Sections \ref{sec:dp} and \ref{sec:exi} establish the existence of an optimizer in a general framework 
(Theorem \ref{thm:dp}) and provide easily verifiable sufficient conditions 
(Lemma \ref{lem:lin} and Theorem \ref{thm:exist}). Sections \ref{sec:amf} and \ref{sec:alm} apply 
these results to prove the existence of an optimal portfolio in models of financial markets with or without friction (Theorem \ref{marquis}).

\section{Dynamic programming}\label{sec:dp}

Let $(\Omega,\F,\FF,P)$ be a complete filtered probability space and let $h$ be a {\em $\F$-normal integrand} on $\reals^n\times\Omega$, i.e.\ an extended real-valued $\B(\reals^n)\otimes\F$-measurable function such that $h(\cdot,\omega)$ is lower semicontinuous (lsc) for all $\omega\in\Omega$; see \cite[Chapter~14]{rw98}. A normal integrand maybe interpreted as a ``random lsc function''. Accordingly, properties of normal integrands are interpreted in the $P$-almost sure sense. For example, a normal integrand $h$ is convex, positively homogeneous, positive on a set $C\subseteq\reals^n$, \ldots if there is an $A\in\F$ with $P(A)=1$ such that $h(\cdot,\omega)$ is convex, positively homogeneous, positive on $C$, \ldots for all $\omega\in A$. This is consistent with the convention of interpreting inequalities etc.\ for random variables in the $P$-almost sure sense. Indeed, random variables may be viewed as normal integrands which do not depend on $x$. 

For a $\sigma$-algebra $\mathcal{G}\subseteq\mathcal{F}$ we denote by $L^0(\Omega,\mathcal{G},P;\mathbb{R}^{d})$ the set of $\mathcal{G}$-measurable $\mathbb{R}^d$-valued random variables, $L^0(\mathcal{G})$ is a shorthand notation for $L^0(\Omega,\mathcal{G},P;\mathbb{R}^d)$ with $d$ being clear from the context, $L^1(\Omega,\mathcal{G},P)$ denotes the set of integrable $\mathbb{R}$-valued random variables. 

We will study the dynamic stochastic optimization problem
\begin{equation}\label{p}\tag{$P$}
\minimize\quad Eh(x) := \int h(x(\omega),\omega)dP(\omega)\quad\text{over $x\in\N$},
\end{equation}
where $\N:=\{(x_t)_{t=0}^T\,|\, x_t\in L^0(\Omega,\F_t,P;\reals^{n_t})$ for given integers $n_t$ such that $n_0+\cdots n_T=n$. We assume throughout the article that there is an $m\in L^1(\Omega,\F,P)$ such that $h\ge m$.

Given a sub-$\sigma$-algebra $\G\subseteq\F$, the {\em conditional expectation} 
$E^\G h$ of $h$ is a $\G$-normal integrand such that
\[
(E^\G h)(x) = E^\G h(x)\quad\forall x\in L^0(\Omega,\G,P;\reals^n).
\]

The next lemma follows from \cite[Corollary 2.2]{chs3}.
\begin{lemma}[\cite{chs3}]\label{lem:ce}
Let $\G\subseteq\F$ be a sigma-algebra. Then $h$ has a well-defined conditional normal integrand $E^\G h$ that is bounded from below by $E^\G m$.
\end{lemma}

We will use the notation $E_t=E^{\F_t}$ and $x^t=(x_0,\ldots,x_t)$ and define extended real-valued functions $h_t,\tilde h_t:\reals^{n_1+\dots+n_t}\times\Omega\rightarrow\ereals$ recursively for $t=T,\ldots,0$ by
\begin{equation}\label{dp}
\begin{split}
\tilde h_T&=h,\\
h_t &= E_t\tilde h_t,\\
\tilde h_{t-1}(x^{t-1},\omega)&=\inf_{x_t\in\reals^{n_t}}h_t(x^{t-1},x_t,\omega).
\end{split}
\end{equation}
In order to guarantee that the above recursion is well defined and 
that optimal solutions exist, we will need to impose appropriate growth conditions on the functions $h_t$. 
Like in \cite{pp12}, our conditions are given in terms of the recession functions of $h_t$. Here, however, 
we are dealing with nonconvex functions so we will use the notion of a horizon function from \cite{rw98} which 
extends the notion of a recession function to the nonconvex case.

We now recall some terminology from \cite{rw98}. A function $g:\reals^n\to\ereals$ is \emph{proper} if it does not take on the value $-\infty$ and it is not identically $+\infty$. The set $\dom g:=\{x\in\reals^n\,|\,g(x)<\infty\}$ is called the {\em effective domain} of $g$. The {\em horizon function} of $g$ is the positively homogeneous function defined by
\[
g^\infty(w) := \lim_{\delta\searrow 0}\inf_{\stackrel{\lambda\in(0,\delta)}{w'\in\uball(w,\delta)}}\lambda 
g(w'/\lambda),
\]
where $\uball(w,\delta)$ denotes the closed ball of radius $\delta$ around $w$.

In some important situations, the horizon function may be expressed as
\begin{equation}\label{w}
g^\infty(w) = \liminf_{\alpha\upto\infty}\frac{g(\alpha w+\bar w)}{\alpha}
\end{equation}
for some $\bar w\in\reals^n$.

{Given a set $C$, we denote by $\delta_C$ the {\em indicator function} of $C$, i.e. 
$\delta_C(x)=0$ if $x\in C$ and $\delta_C(x)=+\infty$ otherwise.} 
\begin{example}\label{ex:horizon}
If $g$ is proper convex lsc function, then, by \cite[Theorem~3.21]{rw98},
\[
g^\infty(w) = \lim_{\alpha\upto\infty}\frac{g(\alpha w+\bar w)}{\alpha} = \sup_{\alpha>0}\frac{g(\alpha w+\bar w)-g(\bar w)}{\alpha}
\]
for any $\bar w\in\dom g$. Expression \eqref{w} holds also for proper lsc functions on the real line with any $\bar w\in\reals$. Indeed, for $w>0$ (analogously for $w<0$), we see from the definition that $g^\infty(w)=(g+\delta_{\reals_+})^\infty(w)$, so the positive homogeneity of $g^\infty$ and the expression in \cite[Theorem~3.26]{rw98} give
\[
g^\infty(w)=wg^\infty(1)= w\liminf_{\alpha\upto\infty}\frac{g(\alpha)}{\alpha}=\liminf_{\alpha\upto\infty}\frac{g(\alpha w)}{\alpha}.
\]
Applying this to the function $g_{\bar w}(w):=g(w+\bar w)$ and using the fact that $g_{\bar w}^\infty=g^\infty$ (see \cite[p.~89]{rw98}) proves the claim.
%
%
\end{example}

For proper convex lsc functions, one has $(g_1+g_2)^\infty=g_1^\infty+g_2^\infty$ whenever $\dom g_1\cap\dom g_2\ne\emptyset$. More generally, we have the following.

\begin{lemma}\label{lem:rec2}
Let $g_1$ and $g_2$ be proper lsc functions with proper horizon functions. Then $(g_1+g_2)^\infty\ge g_1^\infty+g_2^\infty$. If $g_1$ is convex and $g_2$ is satisfies \eqref{w} with some $\bar w\in\dom g_1\cap\dom g_2$, then
\[
(g_1+g_2)^\infty=g_1^\infty+g_2^\infty
\]
and $g_1+g_2$ satisfies \eqref{w} with the same $\bar w$.
\end{lemma}

\begin{proof}
We always have
\begin{align*}
(g_1+g_2)^\infty(w) &= \lim_{\delta\searrow 0}\inf_{\stackrel{\lambda\in(0,\delta)}{w'\in\uball(w,\delta)}}[\lambda g_1(w'/\lambda) + \lambda g_2(w'/\lambda)]\\
&\ge\lim_{\delta\searrow 0}\left[\inf_{\stackrel{\lambda\in(0,\delta)}{w'\in\uball(w,\delta)}}\lambda g_1(w'/\lambda) + \inf_{\stackrel{\lambda\in(0,\delta)}{w'\in\uball(w,\delta)}}\lambda g_2(w'/\lambda)\right]\\
&=g_1^\infty(w)+g_2^\infty(w).
\end{align*}
By shifting the functions if necessary, we may assume that $\bar w=0$ and $g_1(0)=0$; see \cite[p.~89]{rw98}. Then, under the additional assumptions,
\begin{align*}
(g_1+g_2)^\infty(w) &\le \lim_{\delta\searrow 0}\inf_{\lambda\in(0,\delta)}[\lambda g_1(w/\lambda) + \lambda g_2(w/\lambda)]\\
&\le \sup_{\lambda>0}\lambda g_1(w/\lambda) + \lim_{\delta\searrow 0}\inf_{\lambda\in(0,\delta)}\lambda g_2(w/\lambda)\\
&=g_1^\infty(w) + g_2^\infty(w),
\end{align*}
where the last equation follows from convexity of $g_1$; see Example~\ref{ex:horizon}. The above also shows that $g_1+g_2$ satisfies 
\eqref{w} with $\bar{w}=0$.
\end{proof}

By \cite[Exercise 14.54]{rw98}, the function $h^\infty$ defined by $h^\infty(\cdot,\omega)$ is a normal integrand.




\begin{lemma}\label{lem:ip}
Assume that $h_t^\infty(0,x_t)>0$ for all $x_t\ne 0$. Then $\tilde h_{t-1}$ is a normal integrand and
\[
\tilde h_{t-1}^\infty(x^{t-1},\omega)=\inf_{x_t} h_t^\infty(x^{t-1},x_t,\omega).
\]
Moreover, given an $x\in\N$, there is an $\F_t$-measurable $\bar x_t$ such that
\[
\tilde h_{t-1}(x^{t-1}(\omega),\omega) = h_t(x^{t-1}(\omega),\bar x_t(\omega),\omega).
\]
\end{lemma}

\begin{proof}
By \cite[Theorem~3.31]{rw98}, the horizon condition implies that $h_t$ is level-bounded 
locally uniformly in $x^{t-1}$ and that the expression for the horizon function is valid. 
By \cite[Theorem~1.17]{rw98}, the infimum in the definition of 
$\tilde h_{t-1}$ is attained. By \cite[Proposition~14.45(c)]{rw98}, 
the function $p(x,\omega):=h_t(x^{t-1}(\omega),x,\omega)$ is an $\F_t$-measurable 
normal integrand so, by \cite[Theorem~14.37]{rw98}, the minimizer $\bar x_t$ can 
be chosen $\F_t$-measurable. By \cite[Proposition~14.47]{rw98}, $\tilde h_{t-1}$ is a normal integrand.
\end{proof}

With the help of the lemmas above, the following theorem is proved analogously to \cite[Theorem~1]{pp12}. 

\begin{theorem}\label{thm:dp}
Assume that $h_t^\infty(0,x_t)>0$ for all $x_t\ne 0$ whenever $h_t$ is well-defined. Then $h_t$ is well-defined for all $t=T,\ldots,0$ and
\begin{equation}\label{ie}
Eh_t(x^t)\ge \inf\eqref{p}\quad t=0,\ldots,T \quad\forall x\in\N.
\end{equation}
Optimal solutions $x\in\N$ exist and they are characterized by the condition
\[
x_t(\omega)\in\argmin_{x_t}h_t(x^{t-1}(\omega),x_t,\omega)\quad P\text{-a.s.}\quad t=0,\ldots,T,
\]
which is equivalent to having equalities in \eqref{ie}.
\end{theorem}
\begin{proof} 
By recursive application of Lemmas~\ref{lem:ce} and~\ref{lem:ip}, $h_t$ and $\tilde h_t$ are well-defined normal integrands. For $x\in\N$, we have
\[
Eh_t(x^t(\omega),\omega) \ge E\tilde h_{t-1}(x^{t-1}(\omega),\omega) = Eh_{t-1}(x^{t-1}(\omega),\omega)\quad t=1,\ldots,T.
\]
Thus,
\[
Eh(x(\omega),\omega) = Eh_T(x^T(\omega),\omega) \ge Eh_0(x^0(\omega),\omega) \ge E\inf_{x_0\in\reals^{n_0}}h_0(x_0,\omega),
\]
where the inequalities hold as equalities if and only if
\[
h_t(x^t(\omega),\omega)=\tilde h_{t-1}(x^{t-1}(\omega),\omega)\quad P\text{-a.s.}\quad t=0,\ldots,T.
\]
The existence of such an $x\in\N$ follows by applying Lemma~\ref{lem:ip} recursively for $t=0,\ldots,T$.
\end{proof}

The above result is closely related to \cite{evs76} where it was assumed that the sets $\{x\in\reals^n\,|\, h(x,\omega)\le\alpha\}$ are compact for every $\omega\in\Omega$ and $\alpha\in\reals$. In Theorem~\ref{thm:dp}, this has been substituted by the assumption on the horizon functions, which is equivalent to the sets $\{x_t\in\reals^{n_t}\,|\,h_t(x_t,\omega)\le\alpha\}$ being compact; see \cite[Theorem~3.26]{rw98}. As we will see in the following sections, our assumption allows for reformulations that turn into well known no-arbitrage conditions in models of financial economics.


The following lemma gives a sufficient condition for the growth condition in Theorem~\ref{thm:dp}.

\begin{lemma}\label{lem:lin}
We have $h^{\infty}_t(0,x_t)>0$ for all $x_t\neq 0$ provided that
\[
\{x\in \N \mid h^\infty(x)\le 0\}=\{0\}.
\]
\end{lemma}

\begin{proof}
We proceed by induction on $T$. Assume first that the claim holds for the $(T-1)$-period model. 
Applying Lemmas~\ref{lem:ce} and \ref{lem:ip} backwards for $s=T,\ldots,1$, 
we then see that $h_0$ is well defined. Lemmas~\ref{lem:ip} and \ref{lem:rce} give
\begin{align}
&\{x_0\in L^0(\F_0)\,|\, h_0^\infty(x_0(\omega),\omega)\le 0\text{ a.s.}\}\nonumber\\
&\subseteq \{x_0\in L^0(\F_0)\,|\, \tilde h_0^\infty(x_0(\omega),\omega)\le 0\text{ a.s.}\}\nonumber\\
&= \{x_0\in L^0(\F_0)\,|\, \inf_{x_1}h_1^\infty(x_0(\omega),x_1,\omega)\le 0\text{ a.s.}\}\nonumber\\
&= \{x_0\in L^0(\F_0)\,|\, \exists\tilde x\in\N:\ \tilde x_0=x_0,\ h_1^\infty(\tilde x^1(\omega),\omega)\le 0\text{ a.s.}\}\nonumber,
\end{align}
where the last equality follows by applying the last part of Lemma~\ref{lem:ip} to the normal integrand $h^\infty$. Repeating the argument for $t=1,\ldots,T$, we get
\begin{multline}
\{x_0\in L^0(\F_0)\,|\, h_0^\infty(x_0(\omega),\omega)\le 0\text{ a.s.}\}\\
\subseteq \{x_0\in L^0(\F_0)\,|\, \exists\tilde x\in\N: \tilde x_0=x_0,\ h^\infty(\tilde x(\omega),\omega)\le 0 \text{ a.s.}\}=\{0\}\label{mamma}
\end{multline}
Thus $h^\infty_0(x_0)>0$ almost surely for every $x_0\neq 0$, since otherwise there would be a nonzero $x\in L ^0(\F_0)$ with $h^\infty_0(x)\le 0$; this contradicts \eqref{mamma}. For a one-period model, the claim is proved similarly. The same argument with $x_t$ and $\F_t$ in lieu of $x_0$ and $\F_0$ allows us to conclude.
\end{proof}

The following lemma was used in the proof of Lemma~\ref{lem:lin}. 

\begin{lemma}\label{lem:rce} Let $h$ be a normal integrand that is bounded from below by $m\in L^1$. We have $E^\G h^\infty\le(E^\G h)^\infty$ and 
\[
\{x\in L^0(\G)\,|\, (E^\G h)^\infty(x)\le 0\}\subseteq \{x\in L^0(\G)\,|\, h^\infty(x)\le 0\}.
\]
\end{lemma}

\begin{proof} 
The function
\[
\hat h(\lambda,x,\omega)= \begin{cases}
\lambda h(x/\lambda,\omega)\quad &\text{if }\lambda>0\\
h^\infty (x,\omega)\quad &\text{if }\lambda=0\\
+\infty\quad&\text{otherwise}
\end{cases}
\]
is clearly $\B(\reals\times\reals^n)\times\F$-measurable, and it is lower semicontinuous w.r.t. $(\lambda,x)$; this can be deduced as in \cite[Exercise 3.49]{rw98}. Thus $\hat h$ is a normal integrand and, by construction,
\[
h^\infty(x,\omega)=\lim_{\delta\searrow 0} \inf_{\lambda\in[0,\delta],x\in\uball(\bar x,\delta)} \hat h(\lambda,x,\omega).
\]
Let $\bar x\in L^0(\G)$ and $A\in\G$. We have that
\begin{align*}
E[1_A h^\infty(\bar x)]&=E[1_A\lim_{\delta\searrow 0} \inf_{\lambda\in[0,\delta],x\in\uball(\bar x,\delta)}\hat h(\lambda,x)]\\
&=\lim_{\delta\searrow 0} E[1_A\inf_{\lambda\in[0,\delta],x\in\uball(\bar x,\delta)}\hat h(\lambda,x)]\\
&\le\lim_{\delta\searrow 0} \inf_{\lambda\in L^0(\G;[0,\delta]),x\in L^0(\G;\uball(\bar x,\delta))}E[1_A\hat h(\lambda,x)]\\
&=\lim_{\delta\searrow 0} \inf_{\lambda\in L^0(\G;[0,\delta]),x\in L^0(\G;\uball(\bar x,\delta))}E[1_A(E^\G \hat h)(\lambda,x)]\\
&=\lim_{\delta\searrow 0} E[1_A\inf_{\lambda\in[0,\delta]),x\in\uball(\bar x,\delta)} (E^\G \hat h)(\lambda,x)]\\
&\le \lim_{\delta\searrow 0} E[1_A\inf_{\lambda\in(0,\delta)),x\in\uball(\bar x,\delta)} \lambda (E^\G h)(x/\lambda)]\\
&=E[1_A(E^\G h)^\infty(\bar x)],
\end{align*}
which gives $ E^\G h^\infty\le (E^\G h)^\infty$. Here the second and the last equality follow from monotone convergence, and the fourth follows from the interchange rule \cite[Theorem 14.60]{rw98}.


To prove the second claim, let $x\in L^0(\G)$ such that $(E^\G h)^\infty(x)\le 0$. By the first claim, $E^\G h^\infty(x)\le 0$ almost surely so, by the definition of a conditional integrand,
\[
(E^\G h^\infty)(x)=E^\G h^\infty(x).
\]
Since $h^\infty\geq 0$, we have $h^\infty(x)\le 0$ almost surely if and only if $E^\G h^\infty(x)\le 0$ almost surely.
\end{proof}

\section{Existence of solutions}\label{sec:exi}

This section gives the main result of the paper, which is a general existence result for nonconvex dynamic optimization problems. 
This is a nonconvex extension of the existence result in \cite[Theorem~2]{pp12}, 
which in turn extends well-known results in financial mathematics on the existence of optimal 
trading strategies under the no-arbitrage condition. Applications to optimal investment with 
nonconvex utilities will be given in Sections~\ref{sec:amf} and \ref{sec:alm} below.


Recall that a set-valued mapping $S:\Omega\tos\reals^n$ is {\em measurable} if $S^{-1}(O)\in\F$ for every open $O\subset\reals^n$. Here $S^{-1}(O):=\{\omega\in\Omega\,|\,S(\omega)\cap O\ne\emptyset\}$ is the inverse image of $O$.

\begin{theorem}\label{thm:exist}
Assume that there is a measurable set-valued mapping $N:\Omega\tos\reals^n$ such that $N(\omega)$ is a subspace for
each $\omega$,
\[
\{x\in\N \mid h^\infty(x)\le 0\}=\{x\in\N\,|\,x\in N\},
\]
and that $Eh(x+x')=Eh(x)$ for all $x,x'\in\N$ with $x'\in N$ almost surely. Then optimal solutions exist.
\end{theorem}

\begin{proof} 
By \cite[Lemma~5.3]{pen11c}, there exist $\F_t$-measurable set-valued mappings $N_t$ such that $x_t\in L^0(\F_t;N_t)$ if and only if  $\tilde x_t=x_t$ for some $\tilde x\in\N$ with $\tilde x\in N$ and $\tilde x^{t-1}=0$. 
Let
\[
\bar h(x,\omega)=h(x,\omega)+\delta_{\Gamma(\omega)}(x),
\]
where $\Gamma=N^\perp_0\times\dots\times N^\perp_T$ and $N^\perp_t(\omega)$ denotes the orthogonal complement of $N_t(\omega)$.
 
Let us show that for every $x\in\N$, there exists $\bar x\in\N$ such that 
\begin{equation}\label{lehel}
Eh(x)=E\bar h(\bar x). 
\end{equation}
Let $\bar x_0$ be the projection of $x_0$ to $N^\perp_0$. Since $x_0$ and $N_0$ are $\F_0$-measurable, 
$\bar x_0$ is $\F_0$-measurable \cite[Exercise~14.17]{rw98}. By definition of $N_0$, 
there exists $\tilde x\in\N$ with $\tilde x\in N$ and $\tilde x_0= -(x_0-\bar x_0)\in N_0$. By assumption, $Eh(x)=Eh(x+\tilde x)$. Moreover, $(x+\tilde x)_0=\bar x_0\in N^\perp_0$. We may repeat the argument for $t=1,\dots,T$ to construct $\bar x\in\N$ with the claimed properties. Since $\bar h\ge h$ and \eqref{lehel} holds, we have that minimizers of $E\bar h$ minimize $Eh$. 
We can now complete the proof by applying Theorem~\ref{thm:dp} to $\bar h$.

It remains to check the conditions of Lemma \ref{lem:lin} for $\bar{h}$. 
Clearly, $\delta_{\Gamma}^{\infty}=\delta_{\Gamma}$. 
By Lemma~\ref{lem:rec2}, $\bar h^\infty\ge h^\infty+\delta_{\Gamma}$, so 
\[
\{x\in\N \mid \bar h^\infty(x)\le 0\} \subseteq \{x\in\N \mid h^\infty(x)\le 0,\ x\in\Gamma\}.
\]
An element $x$ of the set on the right has both $x_0\in N_0$ and $x_0\in N_0^\perp$ and thus, $x_0=0$. Repeating the argument for $t=1,\dots,T$, we get $x=0$ and thus,
\[
\{x\in\N \mid \bar h^\infty(x)\le 0\}=\{0\}.
\]
By Lemma~\ref{lem:lin}, the existence now follows from Theorem~\ref{thm:dp}.
\end{proof}

Let $h$ be a convex normal integrand and let $\{x\in\N \mid h^\infty(x)\le 0\}$ be a linear space.
Then the condition of Theorem~\ref{thm:exist} is satisfied with 
\[
N(\omega) = \{x\in\reals^n\,|\,h^\infty(x,\omega)\le 0,\ h^\infty(-x,\omega)\le 0\}.
\]
Indeed, this set is linear and, by \cite[Corollary~8.6.1]{roc70a}, 
$h(x+x',\omega)=h(x,\omega)$ for all $x'\in N(\omega)$. We thus recover the existence result of 
\cite[Theorem~2]{pp12}. 
Applications to nonconvex problems will be given in Sections~\ref{sec:amf} and~\ref{sec:alm} below.

\section{An application to mathematical finance}\label{sec:amf}

This section applies Theorem~\ref{thm:exist} to the problem of optimal investment in illiquid financial markets. We consider the discrete-time version of the model in \cite{gura}; see also \cite{doso}.

Let $Z_t$, $t=0,\ldots,T$ be an adapted sequence of $(d-1)$-dimensional random variables representing the marginal price of $d-1$ risky assets in an economy. We imagine that if ``very small'' amounts of asset $i$ were traded then this would take place at the price $Z_t^i$ at time $t$. We assume that the riskless asset in this economy has a price identically $1$ at all times.

As in Carassus and R{\'a}sonyi~\cite{nc}, we model trading strategies by predictable processes 
$\phi=(\phi_t)_{t=1}^T$, where $\phi_t$ denotes the portfolio of risky assets held over $(t-1,t]$. 
Thus $\Delta\phi_t=\phi_t-\phi_{t-1}$ is the portfolio of risky assets bought at time $t-1$ and $\phi_t=\phi_0+\sum_{i=1}^t\Delta\phi_i$. In perfectly liquid markets, the corresponding ``value process'' starting at initial capital $x$ is given by
\[
V^x_t= x+\sum_{i=1}^t\phi_i\cdot \Delta Z_i.
\]
In order to model illiquidity effects, we first rewrite the above as
\[
V_t^x= x - \sum_{i=1}^t\Delta\phi_i\cdot Z_{i-1} + \phi_t\cdot Z_t,
\]
with the convention $Z_{-1}=0$. As usual, the last term is interpreted as the liquidation 
value one would obtain by liquidating the portfolio at time $t$. Under illiquidity, it is more meaningful 
to track the position on the cash account without assuming liquidation at every $t$. We denote the cash 
position held over $(t-1,t]$ by $X^0_t$.

If illiquidity costs at time $t$ are given by an $\F_t$-normal integrand $G_t:\reals^{d-1}\times\Omega\to\mathbb{R}_+$, we have that the change in the cash position at time $t-1$ is 
\[
\Delta X^0_t = \Delta\phi_t\cdot Z_{t-1}-G_{t-1}(\Delta\phi_t)
\]
(recall that $\Delta\phi_t$ is the portfolio of risky assets bought at time $t-1$). Summing up, we get
\[
X^0_t:=X_0^0-\sum_{i=1}^t\Delta\phi_i\cdot Z_{i-1}-\sum_{i=1}^t G_{i-1}(\Delta\phi_i).
\]
Note that the $\Delta\phi_i$ are control variables here while $X^0_t$ is the controlled process. 
We assume that the functions $G_t$ are convex in the first argument and
\begin{eqnarray}
\lim_{\alpha\to\infty}\frac{G_t(\alpha z,\omega)}{\alpha} &\ge& -Z_t(\omega)\cdot z,\quad\forall z\in\reals^{d-1},
\label{kk1}\\
\lim_{\alpha\to\infty}\frac{G_t(\alpha z,\omega)}{\alpha} &>& -Z_t(\omega)\cdot z,\quad\forall z\notin\reals^{d-1}_-.\label{kk2}
\end{eqnarray}
These conditions hold in particular if liquidity costs are \emph{superlinear} in the volume; see Guasoni and  R{\'a}sonyi~\cite{gura}. The above condition allows also for {\em free disposal} of all securities in the sense that the total cost $S_t(z,\omega):=G_t(z,\omega)+Z_t(\omega)\cdot z$ is nondecreasing with respect to the partial order induced by $\reals^{d-1}_-$. This is quite a natural assumption e.g.\ in most securities markets where unit prices are always nonnegative.

We will consider an optimal investment problem of an agent whose financial position is described by a random endowment $W$. We allow both positive and negative values so $W$ can represent financial liabilities as well. The investor's risk preferences are described by a possibly nonconcave utility function $u:\mathbb{R}\to\reals$. We will assume that $u$ is upper semicontinuous, bounded from above and that
\begin{equation}\label{uu1}
\limsup_{\alpha\to\infty}\frac{u(\alpha w,\omega)}{\alpha} < 0\quad\forall w<0.
\end{equation}
For piecewise concave $u$, \eqref{uu1} clearly holds; see \cite{cp} for such a setting.

An application of Theorem~\ref{thm:exist} yields the following existence result; see Example~\ref{ex:proof} 
below for the proof.

\begin{theorem}\label{marquis} 
For an investor with initial capital $X_0^0=z$ and zero initial stock position 
$\phi^j_0=0$, $j=1,\ldots,d-1$ there exists an optimal strategy $\phi^*$ with 
$$\sup_{\phi} Eu(X_T^0(\phi)+W)=Eu(X_T^0({\phi^*})+W).$$
\end{theorem}

\begin{remark} {\rm A similar result has been obtained in Theorem 5.1 of \cite{gura}, in a continuous-time setting. However, in the discrete-time case, Theorem \ref{marquis} above goes much further. In \cite{gura} $u$ was assumed concave while we do not need this assumption here. Also, in \cite{gura} $|G_t(x)|$ was assumed to dominate (constant times) a power function $|x|^{\alpha}$ 
with $\alpha>1$ while here we only need \eqref{kk1} and \eqref{kk2}.} 
\end{remark}



\section{Models with general convex cost functions and portfolio constraints}\label{sec:alm}

This section extends the above existence result to a market model which does not assume the existence of
a cash account a priori. In a market without perfectly liquid asssets it is important to distinguish
between payments at different points in time which are described
by an adapted sequence $c=(c_t)_{t=0}^T$ of claims, each $c_t$ payable at time $t$. 
As in \cite{pen14}, we assume that trading costs are given by an adapted sequence $(S_t)_{t=0}^T$ of 
convex $\F_t$-normal integrands with $S_t(0,\omega)=0$. We also allow for 
portfolio constraints given by an adapted sequence $(D_t)_{t=0}^T$ of closed convex sets, each containing 
the origin. We assume that $D_T=\{0\}$, i.e.\ that the agent liquidates her portfolio at the terminal date.

We will describe the agent's preferences over sequences of payments by a normal integrand $V:\reals^{T+1}\times\Omega\to\ereals$. More precisely, the agent prefers an adapted sequence 
$d^1$ over another $d^2$ if
\[
EV(d^1)<EV(d^2),
\]
i.e.\ $V(d)$ expresses the \emph{disutility} of $d$.
The agent is indifferent between $d^1$ and $d^2$ if the two expectations are equal. We allow $V(\cdot,\omega)$ to be nonconvex but assume that it is bounded from below by an integrable random variable, $V(0,\omega)=0$ and that $V$ is nondecreasing in the sense that if $d^1-d^2\in\reals^{T+1}_-$ then $V(d^1,\omega)\le V(d^2,\omega)$.

\begin{assumption}\label{assV}
The functions $V(\cdot,\omega)$ satisfy \eqref{w} with $\bar w=0$ for all $\omega\in\Omega$, and
\[
V^\infty(d,\omega)\le 0\iff d\in\reals^{T+1}_-.
\]
\end{assumption}

\begin{remark}\label{ooo}
The above conditions on $V$ hold in particular under the extended Inada condition
\[
V^\infty(\cdot,\omega) = \delta_{\reals^{T+1}_-}\quad \forall\omega\in\Omega.
\]
Indeed, since
\[
0\leq V^\infty(d,\omega)\le \liminf_{\alpha\upto\infty}\frac{V(\alpha d,\omega)}{\alpha},
\]
it suffices to note that the equality holds on $\reals^{T+1}_-$ since $V(0,\omega)=0$ and 
$V(\cdot,\omega)$ is nondecreasing in the directions of $\reals^{T+1}_-$.
\end{remark}

The optimal investment problem can now be written as
\begin{equation}\label{alm}
\minimize\quad EV\left(S(\Delta z) + c\right)\quad\ovr\quad z\in\N_D,
\end{equation}
$\N_D:=\{z\in\N\,|\,z_t\in D_t\}$ denotes the set of feasible trading strategies, 
$z_{-1}:=0$ and $S(\Delta z)$ denotes the adapted process 
$(S_t(\Delta z_t(\omega),\omega))_{t=0}^T$ of trading costs. 
Here $z_t$ denotes the portfolio of assets held over $(t,t+1]$. 
In the notation of the previous section $z_t=(X^0_{t+1},\phi_{t+1})$. 

\begin{example}\label{ex:terminal0}
Problems where one is only interested in the level of terminal wealth fit \eqref{alm} with
\[
V(d,\omega)=
\begin{cases}
V_T(d_T,\omega) & \text{if $d_t\le 0$ for $t<T$},\\
+\infty & \text{otherwise},
\end{cases}
\]
where $V_T$ is a normal integrand on $\reals\times\Omega$. Such a function satisfies Assumption~\ref{assV} as soon as
\begin{equation}\label{VT}
\liminf_{\alpha\upto\infty}\frac{V_T(\alpha d_T,\omega)}{\alpha}>0\quad\forall d_T>0,\ \forall\omega\in\Omega.
\end{equation}
Indeed, $V(\cdot,\omega)$ is now the sum of the indicator function of $\reals^T_-\times\reals$ and $g_2(d):=V_T(d_T,\omega)$. Being a lsc proper function on the real line, $V_T(\cdot,\omega)$ automatically satisfies \eqref{w}; see Example~\ref{ex:horizon}. It follows that
\begin{align*}
g_2^\infty(d) &=\lim_{\delta\searrow 0}\inf_{\stackrel{\lambda\in(0,\delta)}{d'\in\uball(d,\delta)}}\lambda V_T(d_T'/\lambda)\\
&= \lim_{\delta\searrow 0}\inf_{\stackrel{\lambda\in(0,\delta)}{d_T'\in\uball(d_T,\delta)}}\lambda V_T(d_T'/\lambda)\\
&= V_T^\infty(d_T)\\
&= \liminf_{\alpha\upto\infty}\frac{V_T(\alpha d_T)}{\alpha}\\
&= \liminf_{\alpha\upto\infty}\frac{g_2(\alpha d)}{\alpha}
\end{align*}
so $g_2$ satisfies \eqref{w} with $w=0$ as well and \eqref{VT} means that $g_2^\infty(d,\omega)\le 0$ iff $d_T\le 0$. Lemma~\ref{lem:rec2} implies that $V^\infty(\cdot,\omega)=\delta_{\reals^T_-\times\reals}+g_2^\infty$, so \eqref{VT} implies Assumption~\ref{assV}.

Problem \eqref{alm} can now be written with explicit budget constraints as
\begin{equation*}
\begin{aligned}
&\minimize\quad & &EV_T(S_T(\Delta z_T)+c_T) \quad\ovr \quad z\in\N_D\\
&\st & &S_t(\Delta z_t)+c_t\le 0,\quad t=0,\ldots,T-1.
\end{aligned}
\end{equation*}
\end{example}

The existence result below involves an auxiliary market model given by
\begin{align*}
S^\infty_t(x,\omega)&=\sup_{\alpha>0}\frac{S_t(\alpha x,\omega)}{\alpha},\\
D^\infty_t(\omega) &=\bigcap_{\alpha>0}\alpha D_t(\omega). 
\end{align*}
By \cite[Theorem~3.21]{rw98}, $S^\infty_t(\cdot,\omega)$ is the horizon function of $S_t(\cdot,\omega)$ 
while by \cite[Theorem~3.6]{rw98}, $D^\infty(\omega)$ coincides with the 
{\em horizon cone} of $D_t(\omega)$ defined in \cite[Section~3.B]{rw98}.

\begin{theorem}\label{thm:alm}
If $\{z\in\N_{D^\infty}\,|\, S^\infty(\Delta z)\le 0\}$ is a linear space, then the infimum in 
\eqref{alm} is attained.
\end{theorem}

\begin{proof}
In order to apply Theorem~\ref{thm:exist}, we write \eqref{alm} as
\[
\begin{aligned}
&\minimize\quad & & EV(d) \quad\ovr \quad z\in\N_D,\ d\in\N\\
&\st & &S(\Delta z)+c\le d,
\end{aligned}
\]
where $d_t$ denotes the total expenditure at time $t$ (alternatively, one could apply the results of \cite{pen99} on composite mappings). 
This fits \eqref{p} with $x=(z,d)$ and $h(x,\omega) = V(d) + \delta_{C(\omega)}(x)$, where
\[
C(\omega) = \{x\,|\, S_t(\Delta z_t,\omega)+c_t(\omega)\le d_t,\ z_t\in D_t(\omega)\}.
\]
By Lemma~\ref{lem:rec2}, Assumption~\ref{assV} and the fact that $0\in C(\omega)$ imply 
$h^\infty(x,\omega) = V^\infty(d,\omega) + \delta^\infty_{C(\omega)}(x)$. 
Since $\delta^\infty_{C(\omega)}=\delta_{C^\infty(\omega)}$, 
\cite[Exercise~3.12]{rw98} and \cite[Exercise~3.24]{rw98} give
\[
h^\infty(x,\omega) =
\begin{cases}
V^\infty(d,\omega) & \text{if $S^\infty_t(\Delta z_t,\omega)\le d_t$,\ $z_t\in D^\infty_t(\omega)$},\\
+\infty & \text{otherwise}.
\end{cases}
\]
Our assumptions on $V$ imply that $V^\infty(d,\omega)\le 0$ if and only if $d\in\reals^{T+1}_-$ so
\begin{align*}
\{x\in\N|\, h^\infty(x)\le 0\text{ a.s.}\}&=\{x\in\N|\, V^\infty(d)\le 0,\ z\in D^\infty,\ S^\infty(\Delta z)\le d\}\\
&=\{x\in\N|\, d\le 0,\ z\in D^\infty,\ S^\infty(\Delta z)\le d\}\\
&=\{x\in\N|\, d=0,\ z\in D^\infty,\ S^\infty(\Delta z)\le 0\},
\end{align*}
where the last equality follows from the fact that $-S^\infty(-\Delta z)\le S^\infty(\Delta z)$ 
(because $S^\infty_t(\cdot,\omega)$ is sublinear) and the assumption that 
$\{z\in\N_{D^{\infty}} |\, S^\infty(\Delta z)\le 0\}$ is linear. Defining
\[
L(\omega) = \{x\in\reals^n\,|\, d=0,\, z_t\in D^\infty(\omega),\ S^\infty_t(\Delta z_t,\omega)\le 0\}
\]
we thus have that the conditions of Theorem~\ref{thm:exist} are satisfied with 
$N(\omega)=L(\omega)\cap[-L(\omega)]$.
\end{proof}

The following example specializes Theorem~\ref{thm:alm} to optimization of terminal utility and market models with a cash account. 

\begin{example}\label{ex:terminal}
Consider again the setting of Example~\ref{ex:terminal0} and assume that there is a perfectly liquid asset, say 
asset $0$, such that, denoting $z=(z^0,\tilde z)$, 
\[
S_t(z,\omega) = z^0 + \tilde S_t(\tilde z,\omega)
\]
and $D_t(\omega)=\reals\times\tilde D_t(\omega)$ for $t=0,\ldots,T-1$ while still $D_T=\{0\}$. The problem can then be written as (fix an adapted $\tilde z$ and minimize over adapted $z^0$)
\begin{equation}\label{e:ca}
\begin{aligned}
&\minimize\quad & &EV_T\left(\sum_{t=0}^T\tilde S_t(\Delta\tilde z_t)+\sum_{t=0}^Tc_t\right)\quad\ovr\quad z\in\N_D.
\end{aligned}
\end{equation}
The linearity condition of Theorem~\ref{thm:alm} means (by Lemma \ref{lem:rec2}) that
\begin{equation}\label{wwh}
\{z\in\N\,|\, \Delta z^0+\tilde S^\infty(\Delta\tilde z)\le 0,\ \tilde z\in\tilde D^\infty,\ z^0_T=0\}
\end{equation}
is a linear space. This holds, in particular, if
\begin{equation}\label{cond:ca}
\tilde S_t^\infty\ge 0\quad\text{and}\quad\tilde S_t(\tilde z)^\infty> 0,\ \forall \tilde{z}\notin\reals^{\tilde J}_-.
\end{equation}
Indeed, the first inequality implies $\Delta z^0\le 0$ and then $z^0=0$ since $z^0_{-1}=0$, by assumption. 
Then, the second inequality implies $\Delta\tilde z_t\le 0$. Since $z_{-1}=0$ and $D_T=\{0\}$, 
by assumption, this can only hold if $z=0$.
\end{example}

The proof of Theorem~\ref{marquis} is now a simple application of the above example.

\begin{example}[Proof of Theorem~\ref{marquis}]\label{ex:proof}
When $V_T(c,\omega)=-u(-c,\omega)$, $\tilde S_t(\tilde z,\omega) = Z_t(\omega)\cdot\tilde z + G_t(\tilde z,\omega)$, $c_0=X_0^0$, $c_T=-W$ and $c_t=0$ for $t=1,\ldots,T-1$, 
$D_t:=\mathbb{R}^d$, $t=0,\ldots,T-1$, $D_T:=\{0\}$,
we can write problem \eqref{e:ca} as
\[
\begin{aligned}
&\maximize\quad & &Eu\left(X_0^0-
\sum_{t=0}^T[Z_t\cdot\Delta\tilde z_t + G_t(\Delta\tilde z_t)] + W\right)\quad\ovr\quad z\in\N_D.
\end{aligned}
\]
This is exactly the problem formulated in Section ~\ref{sec:amf} where the notation $\phi_t=\tilde z_{t-1}$ was used. Conditions \eqref{VT} and \eqref{cond:ca} now become the conditions on $G$ and $u$ given in Section~\ref{sec:amf}. Indeed, since $\tilde S_t(\cdot,\omega)$ are convex, \eqref{cond:ca} becomes \eqref{kk1} and \eqref{kk2}; see Example~\ref{ex:horizon}.
\end{example}

The linearity condition in Theorem~\ref{thm:alm} applies also to the frictionless case. Indeed, in the classical perfectly liquid market model, it becomes the classical {\em no-arbitrage condition}. In nonlinear unconstrained models, is becomes the {\em robust no-arbitrage condition} introduced by Schachermayer~\cite{sch4}; see \cite[Section~4]{pen12} for details. The linearity condition in Theorem~\ref{thm:alm} may hold even without no-arbitrage conditions. One has $\{z\in\N_{D^{\infty}}\,|\, S^\infty(\Delta z)\le 0\}=\{0\}$, for example, when $S$ is such that $S^\infty_t(z,\omega)>0$ for all $z\notin\reals^J_-$. Indeed, $S(\Delta z)\le 0$ then implies $\Delta z_t\le 0$ componentwise, which must hold as an equality since, by assumption, $x_{-1}=0$ and $D_T=\{0\}$. Such a condition holds e.g.\ in limit order markets where the limit order books always have finite depth.



\end{document}